\documentclass[11pt,reqno]{amsart}
\usepackage{amssymb,mathrsfs,mathtools}
\mathtoolsset{showonlyrefs}
\usepackage{hyperref}
\usepackage[shortlabels]{enumitem}
\usepackage{geometry}

\usepackage{tikz,lipsum,lmodern}
\usepackage[most]{tcolorbox}
\usepackage{xcolor}

\numberwithin{equation}{section}

\newtheorem{mainthm}{Theorem}
\newtheorem{thm}{Theorem}[section]

\newtheorem{lem}[thm]{Lemma}

\newtheorem{con}[thm]{Conjecture}
\theoremstyle{remark}

\def\bR {\mathbb{R}}

\def\bZ {\mathbb{Z}}

\def\grad {{\nabla}}

\newcommand{\ud}{\mathrm{d}}
\newcommand{\supp}{\operatorname{supp}}

\title{Calabi--Yau Conjecture for Minimal Hypersurfaces in $\bR^4$ with bounded geometry}
\author{Shrey Aryan}
\author{Alexander D. McWeeney}
\author{Giuseppe Tinaglia}

\begin{document}
\begin{abstract}
The Calabi--Yau conjectures ask whether every complete minimal hypersurface $\Sigma^n\subset\mathbb R^{n+1}$, $n\geq2$, must be unbounded and, more strongly, proper. In this work, we resolve these conjectures for complete, connected, embedded minimal hypersurfaces $ \Sigma^3 \subset \mathbb{R}^4$ with bounded second fundamental form and finite second Betti number \(b_2(\Sigma;\mathbb Z_2)<\infty\).
\end{abstract}
\maketitle

\section{Introduction}
The question of whether completeness forces a minimal hypersurface in $\mathbb{R}^{n+1}$ to be unbounded goes back to Calabi, who in \cite{Ca} conjectured the following:
\begin{con}\label{con:1}
Let $n\geq 2$. A complete minimal hypersurface $\Sigma^{n}\subset \mathbb{R}^{n+1}$ must be unbounded.
\end{con}
For minimal immersions in $\mathbb{R}^3$, Nadirashvili \cite{Na1} constructed a complete minimal immersion of the disk whose image is contained in the unit ball. Later in \cite[p.~360]{Ya2}, Yau asked whether the examples constructed by Nadirashvili are embedded. Colding and Minicozzi answered this question in \cite{CM10} by proving that every complete embedded minimal surface in $\mathbb{R}^3$ with finite topology is proper, consequently proving a stronger version of Conjecture~\ref{con:1}. Their work built on the deep structure theory of embedded minimal surfaces that they had developed in a series of papers \cite{coldingMinicozzi2004I,coldingMinicozzi2004II, coldingMinicozzi2004III,coldingMinicozzi2004IV,coldingMinicozzi2015V}. Subsequently, Meeks and Rosenberg \cite{MRlamination} proved that every complete embedded minimal surface in $\mathbb{R}^3$ with positive injectivity radius is proper, and Meeks, Pérez, and Ros \cite{MPRCY} proved that a complete embedded
minimal surface of finite genus with an infinite number of ends is proper if and only if it has at most two limit ends. See also \cite{MeeksTinagliaChordArc,MeeksTinagliaGeometry,TinagliaZhouRadius}
for related results on constant mean curvature surfaces.

The properness of complete embedded minimal hypersurfaces in higher dimensions is not well understood because of several obstacles, including the failure of the half-space theorem and the absence of tools and properties available in dimension two, such as the Weierstrass representation, the Gauss--Bonnet formula, and parabolicity. The main result of this paper provides a first answer in this direction.
\begin{mainthm}\label{thm:main}
Let $\Sigma^3\subset \bR^4$ be a connected, complete, embedded minimal hypersurface. Assume that $\sup_\Sigma |A_\Sigma|<\infty$ and $b_2(\Sigma;\mathbb{Z}_2)<\infty$, where $b_2(\Sigma;\mathbb{Z}_2)=\dim_{\mathbb Z_2} H_2(\Sigma;\mathbb{Z}_2)$ denotes the second Betti number. Then $\Sigma$ is proper.
\end{mainthm}
In the surface case, the bounded $|A|$ assumption is much more restrictive. Indeed, using Xavier's half-space theorem, Rosenberg \cite{rosenberg2001intersection} showed that every injectively immersed complete minimal surface in $\mathbb{R}^3$ with bounded $|A|$ is proper. In higher dimensions, however, higher-dimensional catenoids provide counterexamples to a higher-dimensional analogue of Xavier's half-space theorem. The finiteness hypothesis on $b_2(\Sigma;\mathbb{Z}_2)$ is natural in view of the nonproper, complete, embedded minimal hypersurfaces in $\mathbb{R}^{n+1}$ constructed by the first two authors for $n\geq 3$, which have unbounded curvature and infinite topology. In particular, when $n=3$, these examples satisfy $b_2(\Sigma;\mathbb{Z}_2)=\infty$. Apart from the nonproper examples described above, several standard families of complete embedded minimal hypersurfaces in $\bR^4$ satisfy
the hypotheses of Theorem \ref{thm:main}. These include the three-dimensional catenoid \cite{blair1975catenoid}, the examples asymptotic to Simons cones in \cite{mazet2014simons}, and embedded minimal hypersurfaces with finitely many ends constructed in \cite{coutant2012deformation} and  \cite{aryanMcWeeney2026calabi}. Our main theorem also applies to every product $\Sigma^2\times\bR$, where $\Sigma^2\subset\bR^3$ is a complete embedded minimal surface with bounded curvature. Indeed, $H_2(\Sigma^2\times\bR;\bZ_2)=0$, even when $\Sigma^2$ has infinite topology. Our main theorem therefore also gives a new proof of properness for complete embedded minimal surfaces in $\mathbb{R}^3$ with bounded curvature, without appealing, as Rosenberg did in \cite{rosenberg2001intersection}, to Xavier's half-space theorem.

The bounded curvature assumption in our main theorem is also subtle. First, if one were to remove this assumption, then by considering minimal hypersurfaces of the form $\tilde\Sigma^2\times \mathbb{R}$, where $\tilde \Sigma^2\subset \mathbb{R}^3$, Theorem~\ref{thm:main} would imply that complete embedded minimal surfaces in $\mathbb{R}^3$ are proper. On the other hand, a folklore example of Meeks, discussed in \cite{alarcon2008density}, may suggest that such a statement is false. Therefore, we expect that removing the bounded curvature assumption would require imposing stronger topological assumptions than those in our main theorem. In particular, we expect that one should also impose $\operatorname{dim}H_1(\Sigma;\mathbb{Z}_2)<\infty$. This condition seems to be consistent with the conjecture of Meeks--Pérez--Ros (cf. \cite{alarcon2008density}), which states that complete embedded minimal surfaces in $\mathbb{R}^3$ with finite genus are proper.

\subsection{Proof Sketch}
Suppose that $\Sigma$ is not proper. By generalizing the arguments first described in \cite[Appendix B]{coldingMinicozzi2004IV} and \cite[Corollary 2.13]{coldingMinicozzi2004II} (see also \cite[Lemma~1.1]{meeks2005uniqueness}) and combining them with the recent results concerning the stable Bernstein problem in \cite{chodoshLi2024stable,chodosh2023stable,chodosh2024stable,catino2024two,mazet2024stable}, the first two authors showed in
\cite[Lemma~2.11]{aryanMcWeeney2026calabi} that after a rigid motion and,
in the slab case, a dilation, $\Sigma$ is proper in an open half-space
or slab and accumulates on the boundary hyperplane $\{x_4=0\}$. The
height function $ h=x_4|_\Sigma$ is therefore positive and harmonic, with $\inf_\Sigma h=0$. For regular values $0<t<\tau$ sufficiently close to zero, bounded curvature
gives $ |\grad_\Sigma h|^2\leq2\Lambda h,$ where $|A_\Sigma|\leq \Lambda$. Thus, the vertical projection is uniformly non-degenerate on $\{h\leq\tau\}$. Note that this is the main place where bounded curvature is assumed. In particular, having a uniform bound on the curvature for heights close to zero suffices for our argument.

Using \cite[Lemma~1.4]{meeks2005uniqueness}, we show that every connected
component $V$ of $\{h<\tau\}$ is a global graph
\begin{align*}
        V=\{(y,u(y)):y\in D\},
\end{align*}
where $u=\tau$ on $\partial D$ and $|Du|<1$.

Since $\{x_4=0\}$ is a limit leaf, there are infinitely many distinct
connected components $V_j$ of $\{h<\tau\}$. At each of the regular levels $\{h=\tau\}$ and $\{h=t\}$, the compact connected components have linearly independent classes in $H_2(\Sigma;\bZ_2)$, and hence each level has only finitely many compact connected components. It follows that one can choose a
connected component $V$ of $\{h<\tau\}$ such that $\partial V$ contains
a connected noncompact component $S\subset\{h=\tau\},$ and $V\cap\{h=t\}$ contains a connected noncompact component $ \Gamma\subset\{h=t\}.$ For the graph function $u$ of $V$, extend $g=\tau-u$ by zero outside $D$. Then, following the proof of \cite[Lemma~5.5]{aryanMcWeeney2026topological}, we can show that
\begin{align*}
        \int_{\bR^3}
        \frac{|\grad g|^2}{1+|y|}
        \,\ud y
        <\infty.
\end{align*}

However, the projected sets $\pi(S)$ and $\pi(\Gamma)$ are connected
and unbounded, and $g$ takes two distinct constant values on them.
They therefore meet every sufficiently large sphere. The Lipschitz
bound gives two fixed-size caps on each such sphere on which the values
of $g$ remain uniformly separated. Applying \cite[Lemma~5.8]{aryanMcWeeney2026topological} then gives a logarithmic lower bound, contradicting the finiteness of the tilt integral above.

\subsection{Acknowledgments}
The first two authors thank their advisors Tobias Colding and William Minicozzi for their support and encouragement. S. Aryan acknowledges support from the Simons Dissertation Fellowship and NSF Grant DMS-2405393. A. McWeeney acknowledges support from the National Science Foundation.

\section{Proof of Theorem~\ref{thm:main}}
Suppose that $\Sigma$ is not proper. By generalizing the arguments first described in \cite[Appendix B]{coldingMinicozzi2004IV} and \cite[Corollary 2.13]{coldingMinicozzi2004II} (see also \cite[Lemma~1.1]{meeks2005uniqueness}) and combining them with the recent results concerning the stable Bernstein problem in \cite{chodoshLi2024stable,chodosh2023stable,chodosh2024stable,catino2024two,mazet2024stable}, the first two authors showed in \cite[Lemma~2.11]{aryanMcWeeney2026calabi} that after a rigid motion and,
in the slab case, a dilation, $\overline\Sigma$ is a minimal
lamination and the inclusion $\Sigma\hookrightarrow U$ is proper,
where
\begin{align*}
        U=\{0<x_4\}
        \qquad
        \text{or}
        \qquad
        U=\{0<x_4<H\},
\end{align*}
and $\{x_4=0\}$ is contained in the limit set of $\Sigma$. Here the closure $\overline \Sigma$ is taken in $\mathbb{R}^4.$ Set $h=x_4|_\Sigma$.  Let
\begin{align*}
        \pi:\bR^4\rightarrow\bR^3,
        \qquad
        \pi(x_1,x_2,x_3,x_4)=(x_1,x_2,x_3).
\end{align*}
Then
\begin{align}
        h>0,
        \qquad
        \inf_\Sigma h=0,
        \qquad
        \Delta_\Sigma h=0.
        \label{eq:height-basic-properties}
\end{align}
Let $\Lambda_*= \max\{1,\sup_\Sigma|A_\Sigma|\}.$ Since $\inf_\Sigma h=0<\sup_\Sigma h,$ Sard's theorem gives regular values $t$ and $\tau$ satisfying
\begin{align}
        0<t<\tau<
        \min\left\{\sup_\Sigma h,\frac{1}{4\Lambda_*}\right\}.
\end{align}
\begin{lem}\label{lem:compact-level-count}
Let $h:M^3\rightarrow\bR$ be a nonconstant harmonic function, where
$M$ is a connected smooth manifold without boundary, and suppose that $B=\dim_{\bZ_2}H_2(M;\bZ_2)<\infty.$ If $c$ is a regular value of $h$, then $\{h=c\}$ has at most $B$ compact connected components.
\end{lem}

\begin{proof}
Suppose that, for some $k\geq1$, the level set $\{h=c\}$ has distinct compact connected components
$C_1,\ldots,C_k$ satisfying
\begin{align}
        [C_1]+\cdots+[C_k]=0
        \qquad
        \text{in }H_2(M;\bZ_2).
\end{align}
Then there is a compact smooth submanifold
$\Omega\subset M$ such that
\begin{align}
        \partial\Omega=C_1\sqcup\cdots\sqcup C_k.
\end{align}
After discarding any component of $\Omega$ with empty boundary, the
maximum principle and $h=c$ on $\partial\Omega$ give $h=c$ on
$\Omega$. Unique continuation then implies that $h$ is constant on
$M$, which is a contradiction. Thus the homology classes of the compact connected components of $\{h=c\}$ are linearly independent. Since $\dim_{\mathbb Z_2}H_2(M;\mathbb Z_2)=B$, there are at most $B$ such components.
\end{proof}
\begin{lem}\label{lem:low-component-global-graph}
Let $V$ be a connected component of $\{h<\tau\}$, and set
\begin{align}
        M=\overline V^{\,\Sigma},
        \qquad
        D=\pi(V).
\end{align}
Here the closure in the above display is taken with respect to $\Sigma.$ Then $\pi:M\rightarrow\overline D$ is a homeomorphism. In particular, there is $u\in C^2(D)\cap C^0(\overline D)$ such that
\begin{align}
        V=\{(y,u(y)):y\in D\},
        \qquad
        u=\tau
        \quad\text{on }\partial D,
\end{align}
and
\begin{align}
        |Du|<1.
        \label{eq:low-component-slope}
\end{align}
\end{lem}

\begin{proof}
For a local choice of unit normal $\nu$,
\begin{align}
        \grad_\Sigma^2h
        =
        \langle\nu,e_4\rangle A_\Sigma
\end{align}
up to the sign convention for $A_\Sigma$. Hence
\begin{align}
        |\grad_\Sigma^2h|\leq\Lambda_*.
\end{align}
Fix $p\in \Sigma$ and set $a=|\grad_\Sigma h(p)|$. If $a>0$, let $\gamma$
be the complete unit-speed geodesic satisfying
\begin{align}
        \gamma(0)=p,
        \qquad
        \gamma'(0)=-\frac{\grad_\Sigma h(p)}{a}.
\end{align}
Since $h>0$,
\begin{align}
        0
        <
        h\left(\gamma\left(\frac{a}{\Lambda_*}\right)\right)
        \leq
        h(p)-\frac{a^2}{2\Lambda_*}.
\end{align}
Therefore
\begin{align}
        |\grad_\Sigma h|^2
        \leq
        2\Lambda_*h.
        \label{eq:height-gradient-estimate}
\end{align}

For $p\in M$ and $v\in T_p\Sigma$, since
\[
|\nabla_\Sigma h(p)|^2
\le
2\Lambda_*h(p)
\le
2\Lambda_*\tau
<
\frac12,
\]
it follows that
\begin{align}
        |d\pi_p(v)|^2
        &=
        |v|^2-dh(v)^2
        \\
        &\geq
        \left(1-|\nabla_\Sigma h(p)|^2\right)|v|^2
        \geq
        \frac12|v|^2.
        \label{eq:vertical-projection-lower}
\end{align}
Hence \(M\) is locally graphical over \(\{x_4=0\}\), up to its boundary, and the local graphing functions have gradient norm less than one.

It remains to show that the local graph structure extends to a global graph. Equivalently, it remains to prove that $\pi:M\to\overline D$ is a homeomorphism.
Note that since $\tau$ is a regular value, $M$ is a smooth closed domain in
$\Sigma$ with
\begin{align}
        \partial M=M\cap\{h=\tau\}.
\end{align}
Moreover, since $\Sigma$ is complete and $M$ is closed in $\Sigma$, the intrinsic length metric on $M$ is complete.

The restriction $\pi|_{\partial M}$ is injective, since two points of
$\partial M$ with the same projection have the same ambient
coordinates. It is also proper: if $K\subset\{x_4=0\}\cong\mathbb R^3$ is compact, then
\begin{align}
        (\pi|_{\partial M})^{-1}(K)
        =
        \partial M\cap\left(K\times\{\tau\}\right),
\end{align}
which is compact because $\Sigma\hookrightarrow U$ is proper. We claim that $\pi|_M$ is injective. Let $p_1,p_2\in M$ satisfy
\begin{align}
        \pi(p_1)=\pi(p_2),
\end{align}
and choose a path $\gamma\subset M$ joining them. Choose a regular
value
\begin{align}
        0<\varepsilon<\min_\gamma h,
\end{align}
and let $M_\varepsilon$ be the connected component of
$M\cap\{h\geq\varepsilon\}$ containing $\gamma$. Indeed, $M_\varepsilon$ is a connected component of the closed set
$M\cap\{h\geq\varepsilon\}$ and is therefore closed in $\Sigma$.
Since $M_\varepsilon\subset\{\varepsilon\leq h\leq\tau\}$ and
$\Sigma\hookrightarrow U$ is proper, it follows that
$\pi|_{M_\varepsilon}$ is proper. Moreover, every connected component
of $\partial M_\varepsilon$ lies in either $\{h=\varepsilon\}$ or
$\{h=\tau\}$, and $\pi$ is injective on each such component. By
\eqref{eq:vertical-projection-lower}, $\pi$ is a submersion. Hence,
\cite[Lemma~1.4]{meeks2005uniqueness} implies that
$\pi|_{M_\varepsilon}$ is injective. Therefore $p_1=p_2$. It follows
that $\pi|_V$ is a diffeomorphism onto $D$.

We now show that $Z:=\pi(\partial M)=\partial D$. First, every point of $\partial M$ is a limit of points of $V$, while the injectivity of $\pi|_M$ gives $Z\cap D=\emptyset$. Therefore, $Z\subset\partial D$. Conversely, let $y_i\in D$ converge to $y\in\partial D$, and let $z_i$ be the first exit point from $D$ on the segment from $y_i$ to $y$. Thus,
\begin{align}
        [y_i,z_i)\subset D,
        \qquad
        z_i\in\partial D,
        \qquad
        z_i\rightarrow y.
\end{align}
The lift of $[y_i,z_i)$ to $V$ has length at most
\begin{align}
        \sqrt{2}\,|y_i-z_i|
\end{align}
by \eqref{eq:vertical-projection-lower}. Completeness therefore gives
a point $q_i\in M$ satisfying
\begin{align}
        \pi(q_i)=z_i.
\end{align}
Since $z_i\notin D$, we have $q_i\in\partial M$, and hence $z_i\in Z$. Since $\pi|_{\partial M}$ is proper, its image $Z$ is closed in $\bR^3$. Therefore, $y\in Z$. Thus
\begin{align}
        \partial D=Z,
        \qquad
        \pi(M)=\overline D.
\end{align}
We next prove continuity of the inverse along $\partial D$. Set
\begin{align}
        p_i=(\pi|_V)^{-1}(y_i),
        \qquad
        q=(\pi|_{\partial M})^{-1}(y).
\end{align}
Since $\pi|_{\partial M}$ is a proper embedding and $z_i\rightarrow y$,
we have $q_i\rightarrow q.$ Moreover, the lifted segment gives
\begin{align}
        d_M(p_i,q_i)
        \leq
        \sqrt{2}\,|y_i-z_i|
        \rightarrow0.
\end{align}
Consequently, $p_i\rightarrow q$. Since the inverse of $\pi$ is smooth
over $D$, this proves that
\begin{align}
        \pi:M\rightarrow\overline D
\end{align}
is a homeomorphism. Define
\begin{align}
        u=h\circ(\pi|_V)^{-1}.
\end{align}
Then $u\in C^2(D)\cap C^0(\overline D)$, $u=\tau$ on
$\partial D$, and
\begin{align}
        \frac{|Du|^2}{1+|Du|^2}
        =
        |\grad_\Sigma h|^2
        \leq
        2\Lambda_*\tau
        <
        \frac12,
\end{align}
which proves \eqref{eq:low-component-slope}.
\end{proof}

\begin{lem}\label{lem:select-noncompact-levels}
There is a connected component $V$ of $\{h<\tau\}$ such that
$\partial V$ has a connected noncompact component $S\subset\{h=\tau\},$
and $V\cap\{h=t\}$ has a connected noncompact component $\Gamma\subset\{h=t\}.$
\end{lem}

\begin{proof}
Fix $y_0\in\{x_4=0\} \cong  \bR^3$. Since $\{x_4=0\}$ is a limit leaf, a lamination
chart centered at $(y_0,0)$ gives distinct points $p_j\in\Sigma$ such
that
\begin{align}
        \pi(p_j)=y_0,
        \qquad
        h(p_j)\rightarrow0.
\end{align}
After discarding finitely many indices, $h(p_j)<t$. Let $V_j$ be the
connected component of $\{h<\tau\}$ containing $p_j$. By
Lemma~\ref{lem:low-component-global-graph}, each $V_j$ meets the
vertical fiber over $y_0$ in at most one point. Thus the $V_j$ are
pairwise distinct.

Set $b_2=b_2(\Sigma;\mathbb{Z}_2)=\dim_{\bZ_2}H_2(\Sigma;\bZ_2).$ By Lemma~\ref{lem:compact-level-count}, the level $\{h=\tau\}$ has at
most $b_2$ compact connected components. Each $\partial V_j$ is
nonempty. Indeed, if $\partial V_j=\emptyset$, then $V_j$ is both open and
closed in $\Sigma$, and hence $V_j=\Sigma$, contradicting
$\{h>\tau\}\neq\emptyset$. Moreover, a connected component of $\{h=\tau\}$ is contained in the boundary of at most one component of $\{h<\tau\}$. Consequently, after discarding finitely many indices, each
$\partial V_j$ has a noncompact connected component $S_j\subset\{h=\tau\}.$

Choose $b_j\in S_j$ and a path in $\overline V_j^{\,\Sigma}$ joining $p_j$ to
$b_j$. Since $h(p_j)<t<\tau=h(b_j),$ the path meets $\{h=t\}$. Let $\Gamma_j$ be the connected component of $\{h=t\}$ containing such an intersection point. Since $t<\tau$,
each $\Gamma_j$ is contained in $V_j$, and hence the $\Gamma_j$ are
pairwise distinct. Lemma~\ref{lem:compact-level-count} implies that
at most $b_2$ of them are compact. Choosing an index outside the resulting finite set of exceptional indices proves the lemma.
\end{proof}

Fix $V$, $S$, and $\Gamma$ as in Lemma~\ref{lem:select-noncompact-levels}, and write $V=\{(y,u(y)):y\in D\}$ as in Lemma~\ref{lem:low-component-global-graph}.

\begin{lem}\label{lem:weighted-base-energy}
Identify $\{x_4=0\}$ with $\bR^3$ and define $g:\bR^3\to \mathbb R$,
\begin{align}
        g(y)=
        \begin{cases}
        \tau-u(y),&y\in D,\\
        0,&y\notin D.
        \end{cases}
\end{align}
Then $g$ is globally $1$-Lipschitz and
\begin{align}
        \int_{\bR^3}
        \frac{|\grad g(y)|^2}{1+|y|}
        \,\ud y
        <\infty.
        \label{eq:finite-weighted-base-energy}
\end{align}
\end{lem}

\begin{proof}
Let $X:\Sigma\rightarrow\bR^4$ denote the position vector. Set
\begin{align}
        E=|\grad_\Sigma h|^2,
        \qquad
        \Phi=\frac12(\tau-h)^2,
        \qquad
        \psi(X)=\frac{1}{\sqrt{1+|X|^2}}.
\end{align}
Then
\begin{align}
        \Delta_\Sigma\Phi=E,
        \qquad
        \Phi=|\grad_\Sigma\Phi|=0
        \quad\text{on }\partial V.
        \label{eq:quadratic-height-test}
\end{align}
Since $\Sigma$ is minimal and three-dimensional,
\begin{align}
        \Delta_\Sigma\psi
        =
        -3\frac{(1+|X^\perp|^2)}{(1+|X|^2)^{5/2}}
        \leq0.
        \label{eq:psi-superharmonic}
\end{align}

Let $\eta_R\in C^\infty_c(B^{\mathbb{R}^4}_{2R}(0))$ be a standard ambient radial cutoff and set $q_R=\eta_R\psi.$ Its support in $\overline V$ is compact because $\supp_{\overline V}q_R \subset \pi^{-1}(\overline D\cap\overline B_{2R}^{\bR^3}(0))$ and $\pi:\overline V\rightarrow\overline D$ is a homeomorphism. Then integration by parts and \eqref{eq:quadratic-height-test} give
\begin{align}
        \int_V q_RE\,\ud\mu_\Sigma
        =
        \int_V\Phi\Delta_\Sigma q_R\,\ud\mu_\Sigma.
        \label{eq:quadratic-weighted-identity}
\end{align}
By \eqref{eq:psi-superharmonic} and the estimates on the cutoff function,
\begin{align}
        \Delta_\Sigma q_R
        \leq
        \frac{C}{R^3}
        \mathbf 1_{B_{2R}(0)\setminus B_R(0)}.
\end{align}
Since $|Du|<1$,
\begin{align}
        \mu_\Sigma\left(V\cap B_{2R}(0)\right)
        \leq
        CR^3.
\end{align}
As $0\leq\Phi\leq\tau^2/2$, equation
\eqref{eq:quadratic-weighted-identity} yields
\begin{align}
        \int_V
        \eta_R\psi E
        \,\ud\mu_\Sigma
        \leq C.
\end{align}
Letting $R\to \infty$ and applying Fatou's lemma gives
\begin{align}
        \int_V
        \frac{E}{\sqrt{1+|X|^2}}
        \,\ud\mu_\Sigma
        <\infty.
\end{align}
On the graph of $u$,
\begin{align}
        E=
        \frac{|Du|^2}{1+|Du|^2},
        \qquad
        \ud\mu_\Sigma=
        \sqrt{1+|Du|^2}\,\ud y.
\end{align}
Since $|Du|<1$ and $0<u<\tau$, it follows that
\begin{align}
        \int_D
        \frac{|Du(y)|^2}{1+|y|}
        \,\ud y
        <\infty.
        \label{eq:finite-Du-energy}
\end{align}
It remains to verify the Lipschitz estimate. If the segment joining
$y,z\in D$ is contained in $D$, then
\begin{align}
        |g(y)-g(z)|
        \leq
        |y-z|.
\end{align}
If it leaves $D$, assume $g(y)\geq g(z)$ and let $\xi$ be the first
point of $\partial D$ on the segment starting at $y$. Since
$g(\xi)=0$,
\begin{align}
        |g(y)-g(z)|
        \leq
        g(y)
        =
        |g(y)-g(\xi)|
        \leq
        |y-\xi|
        \leq
        |y-z|.
\end{align}
The cases in which one or both endpoints lie outside $D$ follow in
the same way. Thus $g$ is globally $1$-Lipschitz. Its weak gradient is
$-Du$ almost everywhere on $D$ and zero almost everywhere outside
$D$. Therefore,
\begin{align}
 \int_{\bR^3}
        \frac{|\grad g(y)|^2}{1+|y|}
        \,\ud y=\int_D
        \frac{|Du(y)|^2}{1+|y|}
        \,\ud y. 
\end{align}
Thus \eqref{eq:finite-weighted-base-energy} follows from \eqref{eq:finite-Du-energy}.
\end{proof}

\begin{proof}[Proof of Theorem~\ref{thm:main}]
Set
\begin{align}
        S_0=\pi(S),
        \qquad
        \Gamma_0=\pi(\Gamma).
\end{align}
Since $\pi:\overline V^{\,\Sigma}\rightarrow\overline D$ is a homeomorphism and $S$ and $\Gamma$ are closed, connected, and noncompact, the sets $S_0$ and $\Gamma_0$ are closed, connected, and noncompact in $\overline D$. Since $\overline D$ is closed in $\bR^3$,
both sets are unbounded. Moreover,
\begin{align}
        g=0
        \quad\text{on }S_0,
        \qquad
        g=\tau-t
        \quad\text{on }\Gamma_0.
\end{align}
The sets $ \{|y|:y\in S_0\}$ and $\{|y|:y\in\Gamma_0\}$ are connected unbounded subsets of $[0,\infty)$. Hence there is
$R_0<\infty$ such that $S_0$ and $\Gamma_0$ meet every sphere
$\partial B_r^{\bR^3}(0)$ with $r\geq R_0$. Set
\begin{align}
        a=\tau-t,
        \qquad
        \rho_0=\frac{a}{8}.
\end{align}
For every sufficiently large $r$, choose
\begin{align}
        q_r\in S_0\cap\partial B_r^{\bR^3}(0),
        \qquad
        p_r\in\Gamma_0\cap\partial B_r^{\bR^3}(0).
\end{align}
If $x$ belongs to the geodesic cap of radius $\rho_0$ centered at
$q_r$ in $\partial B_r^{\bR^3}(0)$, then
\begin{align}
        |x-q_r|
        \leq
        d_{\partial B_r}(x,q_r)
        \leq
        \rho_0.
\end{align}
Since $g(q_r)=0$ and $g$ is $1$-Lipschitz,
\begin{align}
        g(x)
        =
        |g(x)-g(q_r)|
        \leq
        |x-q_r|
        \leq
        \rho_0
        =
        \frac{a}{8}.
\end{align}
Similarly, if $x$ belongs to the geodesic cap of radius $\rho_0$
centered at $p_r$, then
\begin{align}
        |g(x)-a|
        =
        |g(x)-g(p_r)|
        \leq
        |x-p_r|
        \leq
        \rho_0
        =
        \frac{a}{8},
\end{align}
and therefore
\begin{align}
        g(x)
        \geq
        a-\frac{a}{8}
        =
        \frac{7a}{8}.
\end{align} 
Let $c,C>0$ be the constants in \cite[Lemma~5.8]{aryanMcWeeney2026topological}. After increasing $R_0$, we may assume that
\begin{align}
        \rho_0<\frac{r}{100},\text{ and } \frac{Cr}{\rho_0}>1
\end{align}
for every $r\geq R_0$. Therefore,
\cite[Lemma~5.8]{aryanMcWeeney2026topological} gives
\begin{align}
        \int_{\partial B_r^{\bR^3}(0)}
        |\grad_{\partial B_r}g|^2
        \,\ud\mathcal H^2
        \geq
        \frac{ca^2}{\log(Cr/\rho_0)}.
\end{align}
Consequently,
\begin{align}
        \int_{\bR^3}
        \frac{|\grad g(y)|^2}{1+|y|}
        \,\ud y
        &\geq
        \int_{R_0}^\infty
        \frac{1}{1+r}
        \int_{\partial B_r^{\bR^3}(0)}
        |\grad_{\partial B_r}g|^2
        \,\ud\mathcal H^2
        \,\ud r
        \\
        &\geq
        ca^2
        \int_{R_0}^\infty
        \frac{\ud r}
        {(1+r)\log(Cr/\rho_0)}
        =
        \infty.
\end{align}
This contradicts \eqref{eq:finite-weighted-base-energy}. Hence $\Sigma$ is proper.
\end{proof}

\bibliographystyle{alpha}
\bibliography{refs}
\end{document}